\documentclass[11pt]{article}
\usepackage[margin=1in]{geometry}
\usepackage[T1]{fontenc}
\usepackage[utf8]{inputenc}
\usepackage{lmodern}
\usepackage{amsmath,amssymb,amsthm,mathtools}
\usepackage{microtype}
\usepackage{needspace}
\usepackage{tikz}
\usetikzlibrary{arrows.meta}
\usepackage{hyperref}
\hypersetup{colorlinks=true,linkcolor=black,citecolor=black,urlcolor=black,pdftitle={Median components of resonance graphs on surfaces},pdfauthor={Guangfu Wang and Ke Kang}}
\newtheorem{theorem}{Theorem}[section]
\newtheorem{lemma}[theorem]{Lemma}
\newtheorem{proposition}[theorem]{Proposition}
\newtheorem{corollary}[theorem]{Corollary}
\theoremstyle{definition}
\newtheorem{conjecture}[theorem]{Conjecture}
\newtheorem{example}[theorem]{Example}
\newcommand{\F}{\mathcal F}
\newcommand{\M}{\mathcal M}
\newcommand{\FF}{\mathbb F_2}
\newcommand{\CAT}{\operatorname{CAT}(0)}
\newcommand{\Lk}{\operatorname{Lk}}
\title{A Proof of the  Tratnik-Ye
	Medianity Conjecture of Resonance Graphs on Surfaces\thanks{This work was supported by the Natural Science Foundation of Shandong Province (Grant No. ZR2024MA073).}}
\author{Guangfu Wang and Ke Kang\\
	School of Mathematics and Information Sciences, Yantai University\\
	Yantai, Shandong Province, 264005, China\\
	\texttt{gfwang@ytu.edu.cn}}
\date{}
\begin{document}
\maketitle

\begin{abstract}
We prove the Tratnik--Ye conjecture that every connected component of a resonance graph of perfect matchings on a closed surface is median whenever the allowed even faces form a proper subset of all faces. The embedding need not be cellular or strong, and the graph need not be bipartite. We also show that the quotient associated with each face is a tree and that these quotients give an isometric embedding of each component into a Cartesian product of trees. This yields a criterion for the facial parity embedding into a hypercube to be isometric.
\end{abstract}

\section{Introduction}\label{sec:intro}
For a graph $K$, write $V(K)$ and $E(K)$ for its vertex and edge sets. A perfect matching of a finite graph is a set of edges incident with each vertex exactly once. Facial flips connect perfect matchings by replacing one alternating edge class of an even facial cycle with the other. The resulting resonance graphs arose in the study of Kekul\'e structures of benzenoid hydrocarbons. Their early development includes the work of Gr\"undler \cite{Grundler} and the $Z$-transformation graphs of Zhang, Guo, and Chen \cite{ZhangGuoChen}; Zhang's survey \cite{ZhangSurvey} describes the subsequent theory for plane bipartite graphs.

Median and lattice structures are central to this theory. Klav\v zar, \v Zigert, and Brinkmann \cite{KlavzarZigertBrinkmann} proved medianity for catacondensed even ring systems. Lam and Zhang \cite{LamZhang} established a distributive lattice on the perfect matchings of a plane elementary bipartite graph (a connected bipartite graph in which every edge belongs to a perfect matching), with its resonance graph as the covering graph. Propp's work on orientations and matchings, circulated in 1993 and published in revised form in 2025 \cite{Propp}, provides a related lattice framework. Beyond the plane bipartite setting, Tratnik and \v Zigert Pleter\v sek \cite{TratnikPletersek} studied fullerenes and proposed medianity of their resonance components. The distinction between a face coordinate and an isometric hypercube coordinate remains relevant even for plane bipartite graphs; recent results of Che \cite{Che2026} determine when the isometric dimension attains the number of relevant finite faces.

We use the surface conventions of Tratnik and Ye \cite{TratnikYe}. A surface is a connected closed surface, orientable or nonorientable. Let $G$ be a finite simple graph embedded in a surface $\Sigma$. A \emph{face} is a component of $\Sigma\setminus G$, and $\partial f$ denotes the frontier of a face $f$ in $\Sigma$. Write $\F(G)$ for the set of faces. A face is \emph{even} if its boundary is a simple cycle of even length. For an even face $f$, write $E(f)$ and $V(f)$ for the edge and vertex sets of its boundary. No regularity is imposed on the boundaries of other faces. In particular, the embedding is not assumed to be cellular. Here cellular means that every complementary region is an open disc; a strong embedding is a cellular embedding in which every facial boundary is a cycle.

Let $\F\subseteq\F(G)$ consist of even faces, and let $\M(G)$ denote the set of perfect matchings of $G$. The \emph{resonance graph} $R(G;\F)$ has vertex set $\M(G)$, with $M$ adjacent to $N$ precisely when
\[
 M\triangle N=E(f)\qquad\text{for some }f\in\F,
\]
where $\triangle$ denotes symmetric difference. Such an edge is a flip of $f$. The cycle $\partial f$ is then \emph{$M$-alternating}: its edges alternate between membership and nonmembership in $M$. Conversely, flipping an $M$-alternating face gives a perfect matching. The same face remains alternating after the flip.

For a connected graph $H$, let $d_H$ be its graph distance and set
\[
 I_H(x,y)=\{z\in V(H):d_H(x,z)+d_H(z,y)=d_H(x,y)\}.
\]
The graph is \emph{median} if $I_H(x,y)\cap I_H(y,z)\cap I_H(z,x)$ contains exactly one vertex for every triple $x,y,z$. Median graphs originate in Avann's work \cite{Avann}; Bandelt \cite{Bandelt} characterized them as retracts of hypercubes. In particular, every median graph admits an isometric embedding into a hypercube.

Tratnik and Ye \cite{TratnikYe} proved that, when $\F\subsetneq\F(G)$, every resonance component embeds as an induced subgraph of a hypercube. They also gave an example in which their embedding is not isometric. They posed the following conjecture, numbered Conjecture~5.3 in the preprint version of \cite{TratnikYe}.
\begin{conjecture}\label{conj:ty}
Let $G$ be a finite simple graph embedded in a connected closed surface, and let $\F\subsetneq\F(G)$ be a set of even faces. Every connected component of $R(G;\F)$ is a median graph.
\end{conjecture}

Our main result settles Conjecture~\ref{conj:ty} in this generality.
\begin{theorem}\label{thm:main}
Let $G$ be a finite simple graph embedded in a connected closed surface, and let $\F\subsetneq\F(G)$ be a set of even faces. Every connected component of $R(G;\F)$ is a median graph.
\end{theorem}

The proof uses the cube complex whose cubes record simultaneous flips of vertex-disjoint alternating faces. This construction belongs to Turaev's theory of matching complexes \cite{Turaev}. Its local nonpositive curvature follows from the independence of disjoint flips. The main point is to prove simple connectedness. A parity relation at a boundary vertex propagates around an alternating face. Applied to a shortest suitable interval in a closed flip word, this relation produces a pair of equal flips that can be brought together and cancelled. Repeating this operation contracts every closed edge path. The cubical characterization of median graphs \cite{Chepoi} then applies.

There is also a metric refinement of the face quotients introduced by Tratnik and Ye. Fix a component $H$ and a face $f\in\F$. Let $E_f(H)$ consist of the edges of $H$ obtained by flipping $f$. Define $T_f$ by contracting every component of $H-E_f(H)$ to a vertex and retaining the edges induced by $E_f(H)$, with parallel edges replaced by one edge. Let $q_f$ send a matching to the component containing it. The proofs below show that face labels are unique and that no loop occurs in this quotient.
\begin{theorem}\label{thm:trees}
Under the hypotheses of Theorem~\ref{thm:main}, every $T_f$ is a tree, and the map
\[
 q:V(H)\longrightarrow\prod_{f\in\F}V(T_f),\qquad
 q(M)=(q_f(M))_{f\in\F},
\]
is an isometric embedding into the Cartesian product of the trees $T_f$.
\end{theorem}
The Cartesian product has the usual graph structure: an edge changes one coordinate along an edge of its factor. Its distance is the sum of the factor distances. An unused face gives a one-vertex factor; an empty product is a one-vertex graph. Theorem~\ref{thm:trees} strengthens the bipartiteness of the face quotients proved in \cite{TratnikYe}. It also explains why a single binary coordinate for each face can lose distance information.

Section~\ref{sec:parity} establishes the topological and parity facts, including the point that permits arbitrary embeddings. Section~\ref{sec:words} proves cancellation of closed flip words. Section~\ref{sec:cubes} constructs the cube complex and proves Theorem~\ref{thm:main}. Section~\ref{sec:metric} proves Theorem~\ref{thm:trees}, characterizes isometric facial parity embeddings, and records the fullerene consequence and the necessity of excluding at least one face.

\section{Facial incidences and parity}\label{sec:parity}
We first justify the local incidence properties needed below. This avoids imposing restrictions on faces that are never flipped. The \emph{face dual} has one vertex for each face and one edge for each edge of $G$, joining the regions on its two local sides. These regions may coincide, in which case the dual edge is a loop. A \emph{sector} at a vertex is a component of a sufficiently small punctured neighbourhood between consecutive incident half-edges.

\begin{lemma}\label{lem:incidence}
The face dual is connected. If $|\F(G)|\geq2$ and a face $f$ has a simple cycle as its boundary, then $f$ occupies exactly one sector at each boundary vertex and exactly one local side of each boundary edge.
\end{lemma}
\begin{proof}
Choose interior points in two faces. An arc between them in the connected surface can be chosen to avoid the finitely many vertices and to cross the edges transversely at finitely many interior points. The successive complementary regions traversed by this arc give a walk in the face dual. This also applies when $G$ is disconnected.

Now let $C=\partial f$ be a simple cycle. At a vertex $v$ of $C$, every sector belonging to $f$ must be bounded by the two half-edges of $C$ at $v$: any other bounding half-edge would put an edge outside $C$ in the frontier of $f$. At least one such sector exists. If there are two, then $v$ has degree two and both sectors belong to $f$. Both sides of each edge of $C$ at $v$ therefore belong to $f$. The same argument at the next vertex of $C$ forces degree two and two sectors of $f$ there as well. Continuing around $C$ shows that every edge in its boundary has $f$ on both sides. Since $\partial f=C$, all dual edges incident with $f$ would be loops. Dual connectivity would then imply that $f$ is the only face, a contradiction. Thus there is exactly one sector of $f$ at every boundary vertex. The assertion about edge sides follows.
\end{proof}

Unless stated otherwise, henceforth $\F$ is a nonempty proper set of even faces. The cases with no allowed face or no perfect matching are immediate. For $B\subseteq\F$, define
\[
 \partial_2 B=\mathop{\triangle}_{f\in B}E(f),
\]
with $\partial_2\varnothing=\varnothing$. This notation denotes an edge set, with addition interpreted modulo two. We write $\FF=\{0,1\}$ for the field with two elements. The next lemma is the independence form of the parity argument of Tratnik and Ye \cite[preprint, Lemma~3.1]{TratnikYe}.

\begin{lemma}\label{lem:independence}
If $B\subseteq\F$ and $\partial_2 B=\varnothing$, then $B=\varnothing$.
\end{lemma}
\begin{proof}
Assign a value $c(h)\in\FF$ to each face $h$, equal to one for $h\in B$ and zero otherwise. Let an edge $e$ have faces $a,b$ on its two local sides. If $a\ne b$, Lemma~\ref{lem:incidence} implies that the coefficient of $e$ in $\partial_2 B$ is $c(a)+c(b)$. It is zero by assumption, so $c(a)=c(b)$. If the two faces coincide, neither occurrence can be a selected face by Lemma~\ref{lem:incidence}, and the same equality is automatic. Connectedness of the face dual makes the values constant. A face outside $\F$ has value zero; hence every value is zero.
\end{proof}

In particular, distinct faces in $\F$ have distinct boundary edge sets. Thus every edge of $R(G;\F)$ has a unique face label; no convention choosing representative labels is needed.

A \emph{flip word} $w=f_1\cdots f_t$ records a walk $M_0M_1\cdots M_t$ in $R(G;\F)$, with $M_i=M_{i-1}\triangle E(f_i)$. For every face $h\in\F(G)$, write $\varepsilon_w(h)\in\FF$ for its number of occurrences modulo two. Faces outside $\F$ have value zero. If the walk is closed, its total symmetric difference is empty. Lemma~\ref{lem:independence}, applied to the faces with odd occurrence count, gives
\begin{equation}\label{eq:closedparity}
 \varepsilon_w(h)=0\qquad(h\in\F(G)).
\end{equation}
Thus every label in a nonempty closed word occurs at least twice. This extends the cycle formulation in \cite{TratnikYe} to closed walks with repetitions.

Two distinct faces \emph{meet} if their frontiers contain a common vertex of $G$. Allowed faces are \emph{disjoint} when their boundary cycles are vertex-disjoint. Flips of disjoint faces commute: flipping one changes no edge incident with a boundary vertex of the other. In particular, replacing a legal segment $fg$ by $gf$ preserves legality and its endpoints whenever $f$ and $g$ are disjoint.

The following parity statement is the main local observation.
\begin{lemma}\label{lem:propagation}
Let $f\in\F$, and let $u$ be the flip word of a walk from $N$ to $N'$ that does not flip $f$. If $f$ is alternating with respect to both $N$ and $N'$, then $\varepsilon_u(h)$ has the same value for all faces $h\ne f$ meeting $f$.
\end{lemma}
\begin{proof}
Fix $v\in V(f)$. By Lemma~\ref{lem:incidence}, the sector of $f$ at $v$ is unique. List the incident half-edges cyclically as $e_0,e_1,\ldots,e_{d-1}$ so that $e_0,e_{d-1}$ bound this sector. For $1\leq i\leq d-1$, let $r_i$ be the face in the sector between $e_{i-1}$ and $e_i$, and put $\eta_i=\varepsilon_u(r_i)$. A face may occupy several sectors, all with the same parity.

For each intermediate edge $e_i$, $1\leq i\leq d-2$, the parity of its change in membership between $N$ and $N'$ is $\eta_i+\eta_{i+1}$. This formula remains valid if both sectors belong to the same face: that face cannot be allowed, by Lemma~\ref{lem:incidence}, and its parity is zero. Since $f$ is alternating at both endpoints, the matching edge at $v$ belongs to $\{e_0,e_{d-1}\}$ in both $N$ and $N'$. Every intermediate edge therefore belongs to neither matching. Hence $\eta_i=\eta_{i+1}$ for all $1\leq i\leq d-2$. For $d=2$ there is only one $\eta_i$. Thus all faces other than $f$ incident with $v$ have a common parity $c_v$; see Figure~\ref{fig:sectors}.

If $vv'$ is a boundary edge of $f$, the face $r$ on its other side is distinct from $f$ by Lemma~\ref{lem:incidence}. It is incident with both endpoints, giving $c_v=\varepsilon_u(r)=c_{v'}$. Connectedness of $\partial f$ now makes all the local values equal.
\end{proof}

\begin{figure}[tb]
\centering
\begin{tikzpicture}[scale=1.05,line cap=round,font=\small]
\fill[gray!15] (0,0)--(210:1.65) arc[start angle=210,end angle=330,radius=1.65]--cycle;
\foreach \a in {210,150,90,30,-30}{\draw (0,0)--(\a:2);}
\fill (0,0) circle (1.8pt);
\node[above right] at (0.03,0.03) {$v$};
\node[below] at (270:1.2) {$f$};
\node[left] at (210:2) {$e_0$};
\node[left] at (150:2) {$e_1$};
\node[above] at (90:2) {$e_2$};
\node[right] at (30:2) {$e_3$};
\node[right] at (-30:2) {$e_4$};
\foreach \a/\s in {180/1,120/2,60/3,0/4}{\node at (\a:1.17) {$r_{\s}$};}
\node[align=left,anchor=west] at (2.8,0.35) {$e_1,e_2,e_3\notin N\cup N'$};
\node[align=left,anchor=west] at (2.8,-0.3) {$\eta_1=\eta_2=\eta_3=\eta_4$};
\end{tikzpicture}
\caption{The local parity argument at a vertex of degree five. Only the two boundary edges of $f$ can match $v$ at either endpoint of the walk. The unshaded sectors may belong to repeated or noncellular faces.}\label{fig:sectors}
\end{figure}
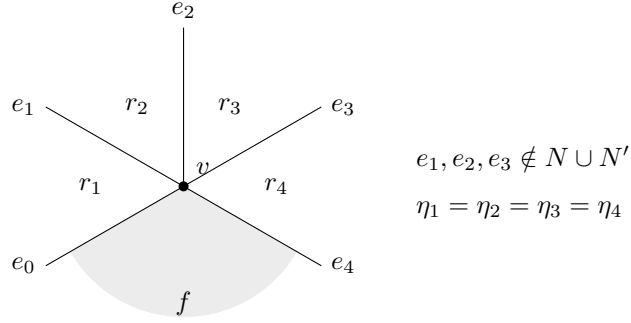

\section{Cancellation of closed flip words}\label{sec:words}
Two elementary operations on legal flip words will suffice: deleting $ff$, and interchanging consecutive labels of disjoint faces. The first deletes an immediate reversal of a flip; the second preserves the endpoints by commutation. Words below are ordinary linear words with a fixed initial matching, and $|U|$ denotes the number of letters in a word $U$.

\Needspace{8\baselineskip}
\begin{proposition}\label{prop:cancellation}
Every closed flip word can be reduced to the empty word by interchanging consecutive disjoint flips and deleting consecutive equal flips.
\end{proposition}
\begin{proof}
Let $w$ be a nonempty closed word, and let $A\subseteq\F$ be the set of labels occurring in it. Because the dual is connected and $A$ is nonempty and proper in $\F(G)$, some $p\in A$ shares a boundary edge with a face $r\notin A$. By \eqref{eq:closedparity}, $p$ occurs at least twice. Two consecutive occurrences, in the linear order of $w$, give a contiguous subword $pUp$ with no $p$ in $U$. The face $r$ meets $p$ and is absent from $U$.

Among all contiguous subwords $fUf$ of $w$ satisfying the following two conditions, choose one with $|U|$ smallest: the label $f$ is absent from $U$, and some face $r\ne f$ meeting $f$ is absent from $U$. The preceding paragraph proves that such a choice exists. We claim that every label in $U$ is disjoint from $f$.

The face $f$ is alternating immediately after the first displayed flip and immediately before the second. Lemma~\ref{lem:propagation} applies to $U$. Since the absent face $r$ has parity zero, every face meeting $f$ has parity zero in $U$. If a label $g$ occurring in $U$ met $f$, it would therefore occur at least twice in $U$. Two consecutive such occurrences would give $gVg$ as a contiguous subword of $U$. The word $V$ contains neither $g$ nor $f$, and $f$ meets $g$. Thus $gVg$ satisfies the same two conditions and $|V|<|U|$, contradicting the choice of $U$. This proves the claim.

Commute the first $f$ past all letters of $U$, and delete the resulting pair $ff$. Every interchange is legal, and the result is a closed word shorter by two. Induction on the length of $w$ completes the reduction.
\end{proof}

The argument does not assume that a loop is already reduced and does not change its basepoint. In particular, no assertion about simple connectedness is used to obtain cancellation.

\section{The cube complex and medianity}\label{sec:cubes}
Fix a component $H$ of $R(G;\F)$ and a matching $M_*\in V(H)$. For $M\in V(H)$, a walk from $M_*$ to $M$ expresses $M\triangle M_*$ as $\partial_2 B$ for a subset $B\subseteq\F$. Lemma~\ref{lem:independence} makes this subset unique. Define $x(M)\in\{0,1\}^{\F}$ to be its characteristic vector, with $x_f(M)$ its entry at $f$, so that
\begin{equation}\label{eq:coordinates}
 M=M_*\triangle\mathop{\triangle}_{\substack{f\in\F\\x_f(M)=1}}E(f).
\end{equation}
The map $x$ is injective. An edge labelled $f$ changes exactly coordinate $f$. Conversely, if two vectors $x(M),x(N)$ differ in exactly coordinate $f$, equation~\eqref{eq:coordinates} gives $M\triangle N=E(f)$, so $M,N$ are adjacent. This is the facial parity realization of the induced hypercube embedding in \cite{TratnikYe}. It is not assumed to preserve distances.

Suppose $S\subseteq\F$ consists of pairwise disjoint faces, each alternating with respect to $M$. Flipping any subset $T\subseteq S$ gives the matching $M\triangle\partial_2T$. In parity coordinates these matchings are exactly the vertices of a face of the Euclidean cube $[0,1]^{\F}$, with variable coordinates $S$ and the other coordinates fixed at $x(M)$. Let $X_H$ be the union of all these faces, including the zero-dimensional ones, as $M$ and $S$ vary. Every face of an included cube is again included. Hence $X_H$ is a finite cube subcomplex of $[0,1]^{\F}$. Cubes are embedded and meet in common faces by this construction. Its one-skeleton is $H$ under the map $x$.

This is the facial version of Turaev's matching complex \cite{Turaev}, whose cubes use vertex-disjoint alternating even cycles. We give the short link verification to make clear that restricting to the specified facial cycles preserves the required local property. The link $\Lk_{X_H}(M)$ has one vertex for each edge at $M$, and a set of link vertices spans a simplex when those edges lie in a common cube. A simplicial complex is \emph{flag} if every finite set of pairwise adjacent vertices spans a simplex.

\begin{lemma}\label{lem:flag}
Every vertex link of $X_H$ is flag.
\end{lemma}
\begin{proof}
At $M$, link vertices correspond to the faces that are $M$-alternating. Two are adjacent precisely when their boundaries are vertex-disjoint, by the construction of $X_H$. A set of pairwise adjacent link vertices therefore consists of pairwise disjoint $M$-alternating faces. Those faces generate a cube, so the set spans a simplex.
\end{proof}

We use the intrinsic piecewise Euclidean metric on the cube complex. For this metric, the link criterion and the Cartan--Hadamard theorem say that a connected finite cube complex is $\CAT$ if it is simply connected and all vertex links are flag; see \cite[Theorems~II.5.20 and II.4.1]{BridsonHaefliger}. The intrinsic metric need not agree with the metric inherited from the ambient Euclidean cube.

\begin{theorem}\label{thm:cat0}
The cube complex $X_H$ is $\CAT$.
\end{theorem}
\begin{proof}
Every loop in a cube complex is homotopic to an edge loop. For an edge loop in $X_H$, Proposition~\ref{prop:cancellation} reduces its flip word to the empty word. Deleting $ff$ removes a backtrack. Interchanging disjoint flips crosses the square generated by those two faces: if a legal segment is $fg$, disjointness implies that both faces were alternating at its initial vertex. Thus every operation is a homotopy with fixed endpoints in $X_H$. Every edge loop is null-homotopic, so $X_H$ is simply connected. Lemma~\ref{lem:flag} and the cubical criterion now apply.
\end{proof}

\noindent\emph{Proof of Theorem~\ref{thm:main}.}
If $\F=\varnothing$, all components are single vertices. Otherwise Theorem~\ref{thm:cat0} applies to every component $H$. The one-skeleton of a $\CAT$ cube complex is a median graph by Chepoi's characterization \cite[Theorem~6.1]{Chepoi}. Since the one-skeleton of $X_H$ is $H$, the result follows.\hfill$\square$

\section{Tree quotients and consequences}\label{sec:metric}
We recall the hyperplane facts used to prove Theorem~\ref{thm:trees}; see Sageev \cite{Sageev} and Chepoi \cite{Chepoi}. In a $\CAT$ cube complex, declare two edges equivalent if they can be joined by a sequence of pairs of opposite edges of squares. Each class is dual to a \emph{hyperplane}. A hyperplane separates the vertices into two halfspaces; an edge has endpoints in different halfspaces precisely when it is dual to that hyperplane. A shortest edge path crosses each hyperplane at most once. Equivalently, the distance between two vertices is the number of hyperplanes separating them.

In $X_H$, opposite edges of a square have the same face label. Thus every hyperplane has a well-defined face label. A label can belong to several hyperplanes, which is why face labels must not be identified with hyperplanes.

\noindent\emph{Proof of Theorem~\ref{thm:trees}.}
Fix $f\in\F$. A path in $H-E_f(H)$ crosses no hyperplane labelled $f$. In particular, the endpoints of an edge labelled $f$ cannot lie in the same component of $H-E_f(H)$. The quotient $T_f$ therefore has no loops. It is connected because $H$ is connected.

We first show that its edges are in bijection with the hyperplanes labelled $f$. Edges dual to the same hyperplane are connected by a sequence of opposite-edge pairs in squares. In each such square, the other two edges have labels different from $f$. The opposite edges consequently join the same pair of components of $H-E_f(H)$, and so define the same quotient edge.

Conversely, let $ab$ and $a'b'$ be edges labelled $f$ that join the same two components, with $a,a'$ in one component and $b,b'$ in the other. If $ab$ is dual to a hyperplane $J$, there are paths from $a$ to $a'$ and from $b$ to $b'$ that cross no $f$-labelled hyperplane. Hence $a,a'$ lie on one side of $J$, and $b,b'$ on the other. The edge $a'b'$ must also be dual to $J$. This proves the bijection.

If $T_f$ had a simple cycle, choose a representative edge of $H$ for each edge of that cycle. Within each contracted component, join the endpoints of successive representatives by a path in $H-E_f(H)$. This produces a closed walk in $H$ that crosses each of the distinct $f$-labelled hyperplanes represented by the cycle exactly once. A closed walk cannot cross a separating hyperplane exactly once. Therefore $T_f$ is a tree.

Let $P$ be a shortest path in $H$ from $M$ to $N$. Its edges cross distinct hyperplanes. After applying $q_f$ and deleting stationary steps, its $f$-labelled edges give distinct edges of the tree $T_f$. This walk cannot repeat a vertex, so it is the unique path from $q_f(M)$ to $q_f(N)$. Its length is the number of $f$-labelled edges of $P$. Summing over all faces proves
\begin{equation}\label{eq:distance}
 d_H(M,N)=\sum_{f\in\F}d_{T_f}\bigl(q_f(M),q_f(N)\bigr).
\end{equation}
Equation~\eqref{eq:distance} also proves injectivity of $q$ and establishes the required isometry.\hfill$\square$

This construction respects medians. Write $\mu_H(M,N,P)$ for the median of three matchings and $\mu_{T_f}$ for the median in a tree. Then
\[
 q_f\bigl(\mu_H(M,N,P)\bigr)
 =\mu_{T_f}\bigl(q_f(M),q_f(N),q_f(P)\bigr).
\]
Indeed, each of the three interval equalities for the median in $H$, together with \eqref{eq:distance}, is a sum of triangle inequalities in the factors. Equality in the sum forces equality in each factor. The image thus lies in the three corresponding tree intervals, whose intersection is the tree median. This is an application of standard median geometry, rather than a separate assumption on matchings.

The Hamming distance between two binary vectors is the number of coordinates in which they differ. We can now characterize exactly when the parity map $x$ from \eqref{eq:coordinates} is isometric. This addresses the limitation of the induced hypercube embedding noted by Tratnik and Ye.
\begin{corollary}\label{cor:isometric}
The facial parity embedding $x:V(H)\to\{0,1\}^{\F}$ is isometric if and only if every $T_f$ has at most one edge.
\end{corollary}
\begin{proof}
Each $f$-flip changes coordinate $x_f$, and no other flip changes it. The proof of Theorem~\ref{thm:trees} therefore gives
\begin{equation}\label{eq:treeparity}
 x_f(M)+x_f(N)\equiv d_{T_f}\bigl(q_f(M),q_f(N)\bigr)\pmod2.
\end{equation}
If every factor has at most one edge, its distances are zero or one, and \eqref{eq:distance} and \eqref{eq:treeparity} show that graph distance equals Hamming distance.

If a factor $T_f$ has at least two edges, it contains two vertices at distance two. Choose preimages $M,N$ under the surjective map $q_f$. Coordinate $f$ contributes zero to their Hamming distance but two to the right side of \eqref{eq:distance}. Each other coordinate contributes at most its corresponding tree distance, by \eqref{eq:treeparity}. Thus the Hamming distance is strictly smaller than $d_H(M,N)$, so $x$ is not isometric.
\end{proof}

The theorem also applies component by component when the original allowed set is not proper. For $\F\subseteq\F(G)$, let
\[
 \F_H=\{f\in\F:\text{$f$ is alternating at some }M\in V(H)\}.
\]
\begin{corollary}\label{cor:inactive}
If $\F_H\subsetneq\F(G)$, then $H$ is median.
\end{corollary}
\begin{proof}
Every edge of $H$ uses a face in $\F_H$. Conversely, flipping a face of $\F_H$ at a matching of $H$, whenever legal, is an edge of the original resonance graph and stays in $H$. Hence $H$ is a component of $R(G;\F_H)$. Apply Theorem~\ref{thm:main}.
\end{proof}

A \emph{fullerene} is a cubic, 3-connected plane graph whose faces are pentagons and hexagons. Its resonance graph allows flips of the hexagonal faces. The following resolves the earlier Conjecture~6.3 of Tratnik and \v Zigert Pleter\v sek \cite{TratnikPletersek} as a special case.
\begin{corollary}\label{cor:fullerene}
Every connected component of the resonance graph of a fullerene is median.
\end{corollary}
\begin{proof}
Regard the plane embedding as an embedding in the sphere. Euler's formula and cubicity give exactly twelve pentagonal faces. The allowed hexagonal faces therefore form a proper subset of all faces, so Theorem~\ref{thm:main} applies.
\end{proof}

Finally, the properness condition cannot be dropped uniformly. The next example is the standard theta-graph obstruction exhibited by Tratnik and Ye \cite[preprint, Figure~1]{TratnikYe}.
\begin{example}\label{ex:allfaces}
Take three internally disjoint $u$--$v$ paths of lengths $1,3,3$ in the sphere. Their unions in pairs bound the three faces, of lengths $4,4,6$. The graph has three perfect matchings: one uses the edge $uv$ and the internal edge of each longer path; each of the other two uses both end edges of one longer path and the internal edge of the other. Any two of these matchings differ on exactly one of the three facial cycles. Allowing all three faces therefore gives the triangle $K_3$, which is not median. Allowing any two faces gives a path on three vertices.
\end{example}

\subsection*{Declaration of AI assistance}
AI-assisted tools were used during revision for proof checking, literature searches, and language editing.

\end{document}